\documentclass{amsart}
\usepackage[T1]{fontenc}
\usepackage{amssymb, amsmath, amsthm}

\usepackage[backend=biber]{biblatex} %
\newtheorem{theorem}[subsection]{Theorem}

\newtheorem{corollary}[subsection]{Corollary}

\newtheorem{remark}[subsection]{Remark}

\newtheorem{example}[subsection]{Example}

\newtheorem{proposition}[subsection]{Proposition}

\newtheorem{lemma}[subsection]{Lemma}

\title{On self-homeomorphisms of the Macías space}

\author{Jhixon Macías}

\address{Jhixon Macías\newline
University of Puerto Rico at Mayaguez, Mayaguez, PR, USA\newline
United States of America}
\email{jhixon.macias@upr.edu}

\keywords{Homeomorphims, Golomb’s topology, Macías topology}

\subjclass[2020]{54A05; 54G05; 54H11}

\begin{document}

\begin{abstract}
In this paper, we study some properties of self-homeomorphisms on the Macías topology over $\mathbb{N}$, and we demonstrate that this space is not topologically rigid.
\end{abstract}

\maketitle
\section{Introduction}\label{section1}
In 1955, H. Furstenberg introduced a topology $\tau_F$ on the integers $\mathbb{Z}$ generated by the collection of arithmetic progressions of the form $a+b\mathbb{Z}$ ($a,b\in \mathbb{Z}, b\geq 1$) and with it presented the first topological proof of the infinitude of prime numbers, see \cite{furstenberg1955infinitude}. A couple of years earlier (1953), M. Brown had introduced a topology $\tau_G$ on the natural numbers $\mathbb{N}$ that turns out to be coarser than Furstenberg's topology induced on $\mathbb{N}$. This latter topology is generated by the collection of arithmetic progressions $a+b\mathbb{N}$ with $a\in \mathbb{N}$ and $b\in\mathbb{N}\cup\{0\}$ such that $\gcd(a,b)=1$. It was not until 1959 that the topology introduced by Brown was popularized by S. Golomb (now known as Golomb's topology) who in \cite{golomb1959connected} proves that Dirichlet's theorem on arithmetic progressions is equivalent to the set of prime numbers $\mathbb{P}$ being dense in the topological space $(\mathbb{N},\tau_G)$. Golomb also presents a topological proof of the infinitude of prime numbers, verifying that Furstenberg's argument also works on $(\mathbb{N},\tau_G)$. A summary of the basic properties of the Furstenberg and Golomb topologies can be found in \cite{steen1978counterexamples}. In 2018, T. Banakh, J. Mioduszewski, and S. Turek \cite{banakh2017continuous} studied some properties of the continuous self-maps and homeomorphisms of the Golomb space. Motivated by this work, in 2020, T. Banakh, D. Spirito, and S. Turek \cite{banakh2019golomb} demonstrated that, indeed, the Golomb space is \textit{topologically rigid}. Topological spaces with a trivial homeomorphism group are called topologically rigid. Later, in 2021, T. Banakh, Y. Stelmakh, and S. Turek \cite{banakh2021kirch} showed that the \textit{Kirch space}, a coarser space than the Golomb space, is also topologically rigid.

In \cite{Jhixon2023} the \textit{Macías topology} $\tau_M\subset\tau_{G}$ is introduced, which is generated by the collection of sets $\sigma_n:=\{m\in\mathbb{N}: \gcd(n,m)=1\}$. Additionally, some of its properties are studied. For example, the topological space $(\mathbb{N},\tau_M)$ does not satisfy the $\mathrm{T}_0$ separation axiom, satisfies (vacuously) the $\mathrm{T}_4$ separation axiom, is ultraconnected, hyperconnected (therefore connected, locally connected, path-connected), not countably compact (therefore not compact), but it is limit point compact and pseudocompact. On the other hand, for $n,m\in\mathbb{N}$, it holds that $\textbf{cl}_{(\mathbb{N},\tau_M)}(\{nm\})=\textbf{cl}_{(\mathbb{N},\tau_M)}(\{n\})\cap \textbf{cl}_{(\mathbb{N},\tau_M)}(\{m\})$ where $\textbf{cl}_{(\mathbb{N},\tau_M)}(\{n\})$ denotes the closure of the singleton set $\{n\}$ in the topological space $(\mathbb{N},\tau_M)$. Furthermore, in \cite{jhixon2024} (usind the topology $\tau_M$), a topological proof of the infinitude of prime numbers is presented (different from the proofs of Furstenberg and Golomb). In the same work, the infinitude of any non-empty subset of prime numbers is characterized in the following sense: if $A\subset \mathbb{P}$ (non-empty), then $A$ is infinite if and only if $A$ is dense in $(\mathbb{N},\tau_M)$, see \cite[Theorem 4]{jhixon2024}.

In this work, motivated by the studies of T. Banakh, S. Turek, J. Mioduszewski, D. Spirito, and Y. Stelmakh, we examine some properties of the self-homeomorphisms of the Macías space \( M(\mathbb{N}):=(\mathbb{N}, \tau_M) \) and show that this space is not topologically rigid. As a result of this study, we characterize the homeomorphisms on \( M(\mathbb{N}) \) (see Theorem \ref{mainth}).

\section{Preliminaries}\label{section2}

Recall that the Macías space $M(\mathbb{N})$ is the topological space $(\mathbb{N},\tau_M)$ where $\tau_M$ is generated by the collection of sets $\sigma_n := \{ m \in \mathbb{N} : \gcd(n, m) = 1 \}$. Here, $\gcd(n, m)$ denotes the \textit{greatest common divisor of $n$ and $m$}. For each $x \in \mathbb{N}$, we denote by $x \cdot \mathbb{N}$ the set $\{ x \cdot n : n \in \mathbb{N} \}$. A topological space $X$ is hyperconnected (ultraconnected) if there are no two non-empty disjoint open (closed) sets. The following theorem will be used implicitly in Section \ref{section 3}.
\begin{theorem}[\cite{Jhixon2023},\cite{jhixon2024},\cite{MACIAS2024109070}]\label{thpreli} In $M(\mathbb{N})$, the following propositions hold:

\begin{enumerate}
\item The space $M(\mathbb{N})$ is hyperconnected and ultraconnected.
\item For any $n, m \in \mathbb{N}$, $\sigma_{n \cdot m } = \sigma_n \cap \sigma_m$. Consequently, $\sigma_{n^\alpha} = \sigma_n$ for any positive integer $\alpha$.
\item For each $n \in \mathbb{N}$, $\sigma_n$ is infinite, since $$\sigma_n = \displaystyle\bigcup_{\substack{1 \leq m \leq n \\ \gcd(m, n) = 1}} n \cdot \left( \mathbb{N} \cup \{0\} \right) + m.$$
\item For any prime number $p$, $\mathbb{N} \setminus \sigma_p = p \cdot \mathbb{N}$.
\item The set $\{1\}$ is dense in $M(\mathbb{N})$ ($1 \in \sigma_n$ for all $n \in \mathbb{N}$). Therefore, $\mathbf{cl}_{M(\mathbb{N})}(\{1\}) = \mathbb{N}$.
\item If $n > 1$, then $$\mathbf{cl}_{M(\mathbb{N})}(\{n\}) = \bigcap_{i=1}^k p_i \cdot \mathbb{N} = p_1 \cdot p_2 \cdots p_k \cdot \mathbb{N},$$ where each $p_i$ is a prime factor in the prime decomposition of $n$.
\item For any $n, m \in \mathbb{N}$, $\mathbf{cl}_{M(\mathbb{N})}(\{n \cdot m\}) = \mathbf{cl}_{M(\mathbb{N})}(\{n\}) \cap \mathbf{cl}_{M(\mathbb{N})}(\{m\})$. Consequently, $\mathbf{cl}_{M(\mathbb{N})}(\{n^\alpha\}) = \mathbf{cl}_{M(\mathbb{N})}(\{n\})$ for any positive integer $\alpha$.
\end{enumerate}
\end{theorem}

\section{Self-homeomorphisms on $M(\mathbb{N})$}\label{section 3}

A topological space $X$ is defined to be \textit{topologically rigid} if its homeomorphism group is trivial. As mentioned in the introduction, T. Banakh, D. Sirito, and S. Turek demonstrate that the Golomb space is topologically rigid. Later, T. Banakh, Y. Stelmakh, and S. Turek show that the Kirch space is topologically rigid. Motivated by this, since the Macías space is coarser than the Golomb space, we ask: Is the Macías space topologically rigid? In this section, we will show that $M(\mathbb{N})$ is not topologically rigid, and we will study some of the properties of its self-homeomorphisms.

Let us begin with an example.

\begin{example}\label{exam1} Let $h: M(\mathbb{N}) \to M(\mathbb{N})$ be given by
\begin{equation*}
    h(n) = \left\{ \begin{array}{lcc} 4 & \text{if} & n = 2 \\ \\ 2 & \text{if} & n = 4 \\ \\ n & \text{otherwise.} \end{array} \right.
\end{equation*}
Note that $h$ is injective and surjective. Let's check that it is continuous. Consider an arbitrary basic open set, say $\sigma_k = \{ x \in \mathbb{N} : \gcd(x, k) = 1 \}$. If $k$ is even, $2,4 \notin \sigma_k$. Then,
\begin{equation*}
    h^{-1}(\sigma_k) = \{ n : \gcd(f(n), k) = 1 \} = \{ n : \gcd(n, k) = 1 \} = \sigma_k.
\end{equation*}
If $k$ is odd, $2, 4 \in \sigma_k$. Then,
\begin{equation*}
    h^{-1}(\sigma_k) = \{ n : \gcd(f(n), k) = 1 \} = \{ n : \gcd(n, k) = 1 \} = \sigma_k.
\end{equation*}
In either case, $h^{-1}(\sigma_k)$ is open in $M(\mathbb{N})$, which implies that $h$ is continuous. Therefore, $h$ is a homeomorphism. 
\end{example}

The previous example already indicates that the homeomorphism group of $M(\mathbb{N})$ is not trivial. From this, the following proposition follows.

\begin{proposition}
The space $M(\mathbb{N})$ is not topologically rigid.
\end{proposition}

Now, a series of lemmas that will help us characterize the homeomorphisms of $M(\mathbb{N})$.

\begin{lemma}\label{lem1h}
Let $h: M(\mathbb{N}) \to M(\mathbb{N})$ be a homeomorphism. Then $h(1) = 1$.
\end{lemma}
\begin{proof}
Suppose that $h(1) > 1$. Then we can write $h(1) = p_1^{\alpha_1} \cdot p_2^{\alpha_2} \cdots p_r^{\alpha_r}$. Thus,
\begin{equation*}
    \mathbf{cl}_{M(\mathbb{N})}(\{ h(1) \}) = p_1 \cdot p_2 \cdots p_r \cdot \mathbb{N}.
\end{equation*}
On the other hand, since $h$ is a homeomorphism and $1$ is dense in $M(\mathbb{N})$,
\begin{equation*}
    \mathbf{cl}_{M(\mathbb{N})}(\{ h(1) \}) = \mathbf{cl}_{M(\mathbb{N})}(h(\{1\})) = h(\mathbf{cl}_{M(\mathbb{N})}(\{1\})) = h(\mathbb{N}) = \mathbb{N},
\end{equation*}
and thus, $p_1 \cdot p_2 \cdots p_r \cdot \mathbb{N} = \mathbb{N}$, which is a contradiction. Therefore, $h(1) = 1$.
\end{proof}

\begin{lemma}\label{lem2h}
Let $h: M(\mathbb{N}) \to M(\mathbb{N})$ be a homeomorphism. Then
\begin{equation*}
    \mathbf{cl}_{M(\mathbb{N})}(\{h(nm)\}) = \mathbf{cl}_{M(\mathbb{N})}(\{h(n) \cdot h(m)\}),
\end{equation*}
for all $n, m \in \mathbb{N}$.
\end{lemma}
\begin{proof}
   This proof is based on the properties of the closure operator on $M(\mathbb{N})$ and the properties of homeomorphisms with respect to a closure operator in general. Note then that
   \begin{equation*}
       \begin{split} \mathbf{cl}_{M(\mathbb{N})}(\{h(n \cdot m)\}) &= h(\mathbf{cl}_{M(\mathbb{N})}(\{n \cdot m\})) \\
        &= h(\mathbf{cl}_{M(\mathbb{N})}(\{n\}) \cap \mathbf{cl}_{M(\mathbb{N})}(\{m\})) \\
        &= h(\mathbf{cl}_{M(\mathbb{N})}(\{n\})) \cap h(\mathbf{cl}_{M(\mathbb{N})}(\{m\})) \\
        &= \mathbf{cl}_{M(\mathbb{N})}(\{h(n)\}) \cap \mathbf{cl}_{M(\mathbb{N})}(\{h(m)\}) \\
        &= \mathbf{cl}_{M(\mathbb{N})}(\{h(n) \cdot h(m)\}),
       \end{split}
   \end{equation*}
   which completes the proof.
\end{proof}

Lemma \ref{lem2h} indicates that \textit{within} the closure operator $\mathbf{cl}_{M(\mathbb{N})}$, the homeomorphism $h$ is completely multiplicative. This fact, along with Lemma \ref{lem1h}, allows us to prove the following lemma.

\begin{lemma}\label{lem3h}
Let $h: M(\mathbb{N}) \to M(\mathbb{N})$ be a homeomorphism. Let $p$ be a prime number. Then $h(p) = h(p^\gamma)^\alpha$ where $h(p^\gamma)$ is prime.
\end{lemma}
\begin{proof}
    Since $p > 1$, by Lemma \ref{lem1h}, we can write $h(p) = p_1^{\alpha_1} \cdot p_2^{\alpha_2} \cdots p_r^{\alpha_r}$, where each $p_i$ is a prime number. Then, note that
    \begin{equation*}
        \mathbf{cl}_{M(\mathbb{N})}(\{h(p)\}) = \mathbf{cl}_{M(\mathbb{N})}(\{p_1 \cdot p_2 \cdots p_r\}).
    \end{equation*}
Now, since $h$ is surjective, for each $i$ we can write $h(a_i) = p_i$ for some $a_i$. Thus, using Lemma \ref{lem2h}, we have
\begin{equation*}
\mathbf{cl}_{M(\mathbb{N})}(\{h(p)\}) = \mathbf{cl}_{M(\mathbb{N})}(\{h(a_1 \cdot a_2 \cdots a_r)\}).
\end{equation*}
Using this last equation, the fact that $p$ is prime, the fact that $h$ is a homeomorphism, and writing $a_1 \cdot a_2 \cdots a_r = q_1^{\beta_1} \cdot q_2^{\beta_2} \cdots q_k^{\beta_k}$ where each $q_i$ is a prime number, we get
\begin{equation*}
    h(p \cdot \mathbb{N}) = h(q_1 \cdot q_2 \cdots q_k \cdot \mathbb{N}).
\end{equation*}
In particular, $h(p) = h(q_1 \cdot q_2 \cdots q_k \cdot a)$ for some positive integer $a$. Then, since $h$ is injective, it must be that $p = q_1 \cdot q_2 \cdot q_k \cdot a$. Since $p$ is prime, we conclude that $a = k = 1$ and $q_1 = p$. Thus, each $p_i = h(p^{\gamma_i})$. Without loss of generality, consider $p_1 = h(p^{\gamma_1})$. Then, using the fact that $h$ is a homeomorphism, that each $p_i$ is prime, and Lemma \ref{lem2h}, we obtain
\begin{equation*}
h(p^{\gamma_1}) \cdot \mathbb{N} = \mathbf{cl}_{M(\mathbb{N})}(\{h(p^{\gamma_1})\}) = \mathbf{cl}_{M(\mathbb{N})}(\{h(p)\}) = \mathbf{cl}_{M(\mathbb{N})}(\{h(p^{\gamma_i})\}) = h(p^{\gamma_i}) \cdot \mathbb{N},
\end{equation*}
which implies that $h(p^{\gamma_1}) = h(p^{\gamma_2}) = \dots = h(p^{\gamma_r})$ and moreover, $\gamma_1 = \gamma_2 = \dots = \gamma_r := \gamma$. Therefore, $p_1 = p_2 = \dots = p_r$. Thus, we can write
\begin{equation*}
    h(p) = q^\alpha,
\end{equation*}
where $q = h(p^\gamma)$ is prime and $\alpha = \alpha_1 + \alpha_2 + \dots + \alpha_r$.
\end{proof}
\begin{remark}\label{rem1h}
Regarding Lemma \ref{lem3h}, note that $q = h(p^\gamma)$ is unique for each $p$. Moreover, for each power of $p^n$, there exists a unique power $q^m$ such that $h(p^n) = q^m$. This is because
\begin{equation*}
\mathbf{cl}_{M(\mathbb{N})}(\{h(p^n)\}) = \mathbf{cl}_{M(\mathbb{N})}(\{h(p)\}) = \mathbf{cl}_{M(\mathbb{N})}(\{q^\alpha\}) = \mathbf{cl}_{M(\mathbb{N})}(\{q\}) = q \cdot \mathbb{N}.
\end{equation*}
Thus, the only prime that divides $h(p^n)$ is $q$, meaning that $h(p^n) = q^m$ for some $m$. Now, if $p^*$ is a prime such that $h(p^*) = q^l$ for some $l$, then $p^* = p$. This follows because in this case
\begin{equation*}
\mathbf{cl}_{M(\mathbb{N})}(\{h(p^*)\}) = \mathbf{cl}_{M(\mathbb{N})}(\{q\}) = \mathbf{cl}_{M(\mathbb{N})}(\{p\}),
\end{equation*}
which would imply that $h(p^* \cdot \mathbb{N}) = h(p \cdot \mathbb{N})$, and therefore $p^* \cdot \mathbb{N} = p \cdot \mathbb{N}$, so $p^* = p$. In summary, for each power $q^a$, there exists a power $p^b$ such that $h(p^b) = q^a$, and vice versa.
\end{remark}



\begin{lemma}\label{lem4h}
Let $h:M(\mathbb{N}) \to M(\mathbb{N})$ be a homeomorphism. Let $p$ be a prime number. Then $h(\sigma_p) = \sigma_{h(p)}$.
\end{lemma}
\begin{proof}
Note that $\mathbb{N} \setminus \sigma_p = \mathbf{cl}_{M(R)}(\{p\})$ (this holds for any prime number). Now, using the fact that $h$ is a homeomorphism, we have
\begin{equation*}
\begin{split}
x \notin h(\sigma_p) \Longrightarrow x \in \mathbb{N} \setminus h(\sigma_p) &\Longrightarrow x \in h(\mathbb{N} \setminus \sigma_p) \\
&\Longrightarrow x \in h(\mathbf{cl}_{M(\mathbb{N})}(\{p\})) \\
&\Longrightarrow x \in \mathbf{cl}_{M(\mathbb{N})}(\{h(p)\}) \\
&\Longrightarrow x \notin \sigma_{h(p)}.
\end{split}
\end{equation*}
Thus, $\sigma_{h(p)} \subset h(\sigma_p)$. On the other hand, by Lemma \ref{lem3h}, $h(p)$ is of the form $h(p^\gamma)^\alpha$ with $h(p^\gamma)$ prime, so $\sigma_{h(p)} = \sigma_{h(p^\gamma)}$. Hence,
\begin{equation*}
\begin{split}
x \notin \sigma_{h(p)} \Longrightarrow x \notin \sigma_{h(p^\gamma)} \Longrightarrow x \in \mathbb{N} \setminus \sigma_{h(p^\gamma)} &\Longrightarrow x \in \mathbf{cl}_{M(\mathbb{N})}(\{h(p^\gamma)\}) \\
&\Longrightarrow x \in h(\mathbf{cl}_{M(\mathbb{N})}(\{p\})) \\
&\Longrightarrow x \in h(\mathbb{N} \setminus \sigma_p) \\
&\Longrightarrow x \in \mathbb{N} \setminus h(\sigma_p) \\
&\Longrightarrow x \notin h(\sigma_p).
\end{split}
\end{equation*}
Thus, $h(\sigma_p) \subset \sigma_{h(p)}$. Therefore, $h(\sigma_p) = \sigma_{h(p)}$.
\end{proof}

Lemma \ref{lem3h} can be extended as follows.

\begin{lemma}\label{lem5h}
Let $h:M(\mathbb{N}) \to M(\mathbb{N})$ be a homeomorphism. Let $n > 1$ be a positive integer. Then $h(n) = h(p_1^{\gamma_1})^{\alpha_1} \cdot h(p_2^{\gamma_2})^{\alpha_2} \cdots h(p_r^{\gamma_r})^{\alpha_r}$ where each $h(p_i^{\gamma_i})$ is prime and the $p_i$ are the primes dividing $n$.
\end{lemma}
\begin{proof}
For each $p_i$, Lemma \ref{lem3h} allows us to write
\begin{equation*}
h(p_i \cdot \mathbb{N}) = h(p_i^{\gamma_i}) \cdot \mathbb{N}.
\end{equation*}
where each $h(p_i^{\gamma_i})$ is prime. Then, fixing $i = 1$, we can write
\begin{equation*}
    h(p_1 \cdot p_2 \cdots p_r) = h(p_1^{\gamma_1}) \cdot k
\end{equation*}
for some integer $k$. Then, by combining Lemma \ref{lem2h} with Lemma \ref{lem3h}, we can write
\begin{equation*}
h(p_1^{\gamma_1}) \cdot h(p_2^{\gamma_2}) \cdots h(p_r^{\gamma_r}) \cdot \mathbb{N} = \mathbf{cl}_{M(\mathbb{N})}(h(p_1 \cdot p_2 \cdots p_r)) = h(p_1^{\gamma_1}) \cdot m \cdot \mathbb{N}.
\end{equation*}
where $m$ is the product of the primes dividing $k$. Therefore,  $h(n) = h(p_1^{\gamma_1})^{\alpha_1} \cdot h(p_2^{\gamma_2})^{\alpha_2} \cdot h(p_r^{\gamma_r})^{\alpha_r}$.
\end{proof}

The Lemma \ref{lem5h} allows us to extend Lemma \ref{lem4h} as follows.

\begin{lemma}\label{lem6h}
Let $h:M(\mathbb{N}) \to M(\mathbb{N})$ be a homeomorphism. Let $n$ be a positive integer. Then $h(\sigma_n) = \sigma_{h(n)}$.
\end{lemma}
\begin{proof}
If $n = 1$, applying Lemma \ref{lem1h}, we have $h(\sigma_1) = \sigma_1$. If $n > 1$, write $n = p_1^{\beta_1} \cdot p_2^{\beta_2} \cdots p_r^{\beta_r}$. Then, by Lemma \ref{lem5h}, $h(n) = h(p_1^{\gamma_1})^{\alpha_1} \cdot h(p_2^{\gamma_2})^{\alpha_2} \cdot h(p_r^{\gamma_r})^{\alpha_r}$. Now, using Lemma \ref{lem4h}, we have
\begin{equation*}
\sigma_{h(n)} = \bigcap_{i=1}^r \sigma_{h(p_i^{\gamma_i})} = \bigcap_{i=1}^r \sigma_{h(p_i)} = \bigcap_{i=1}^r h(\sigma_{p_i}) = h\left( \bigcap_{i=1}^r \sigma_{p_i} \right) = h\left( \bigcap_{i=1}^r \sigma_{p_i^{\alpha_i}} \right) = h(\sigma_n).
\end{equation*}
This completes the proof.
\end{proof}

\begin{corollary}\label{cor1h}
Let $h:M(\mathbb{N}) \to M(\mathbb{N})$ be a homeomorphism. Then, $h(\sigma_{n \cdot m}) = \sigma_{h(n \cdot m)} = \sigma_{h(n) \cdot h(m)}$ for all $n, m \in \mathbb{N}$.
\end{corollary}
    

Note that if $h: M(\mathbb{N}) \to M(\mathbb{N})$ is a bijection such that $h(\sigma_n) = \sigma_{h(n)}$ for all $n \in \mathbb{N}$, then $h$ is a homeomorphism. Indeed, let $n \in \mathbb{N}$, then $n = h(m)$ for some $m$. Therefore,
\begin{equation*}
    h^{-1}(\sigma_n) = h^{-1}(\sigma_{h(m)}) = h^{-1}(h(\sigma_m)) = \sigma_m.
\end{equation*}
Thus, we can establish the following theorem.
\begin{lemma}\label{th1h}
Let $h: M(\mathbb{N}) \to M(\mathbb{N})$ be a bijection. Then $h$ is a homeomorphism if and only if $h(\sigma_n) = \sigma_{h(n)}$ for all $n \in \mathbb{N}$.
\end{lemma}

\begin{example}Lemma \ref{th1h} allows us to easily show that $h: M(\mathbb{N}) \to M(\mathbb{N})$ given by
\begin{equation*}
    h(n) = \left\{ \begin{array}{lcc} 2 & \text{if} & n = 3 \\ \\ 3 & \text{if} & n = 2 \\ \\ n &&\textit{otherwise.} \end{array} \right.
\end{equation*}
is not a homeomorphism. If we assume that $h$ is a homeomorphism, then
\begin{equation*}
    h(\sigma_3) = \sigma_2,
\end{equation*}
but then $4 = h(4) \in \sigma_2$, which is a contradiction. Note also that Lemma \ref{lem3h} allows us to reach the same conclusion.
\end{example}


For each $x \in \mathbb{N}$, we denote by $x^\mathbb{N}$ the set $\{x^n : n \in \mathbb{N}\}$. For each prime number $p$, let $\sigma^p: A_p \subset p^\mathbb{N} \to A'_p \subset p^\mathbb{N}$ be a bijection. Let $\sigma^{-p}$ be the inverse of $\sigma^p$. Lemma \ref{th1h} allows us to construct homeomorphisms on $M(\mathbb{N})$. See the following examples.

\begin{example}\label{th2h}
Let $p$ be a prime number. Then $h: M(\mathbb{N}) \to M(\mathbb{N})$ given by
\begin{equation*}
    h(n) = \left\{ \begin{array}{lcc} \sigma^p(n) & \text{if} & n \in A_p  \\ \\ \sigma^{-p}(n) & \text{if} & n \in A'_p  \\ \\ n & \text{if} & n \notin (A_p \cup A'_p) \end{array} \right.
\end{equation*}
is a homeomorphism.
\end{example}
\begin{proof}
Clearly, $h$ is a bijection. Now, let $\sigma_k$ be arbitrary but fixed. Note that if $k \notin (A_p \cup A'_p)$, then $h(k) = k$ and thus $\sigma_{h(k)} = \sigma_k$. If $k \in A_p$, then $k = p^\alpha$ for some $\alpha$ and $\sigma_{h(k)} = \sigma_{\sigma^p(p^\alpha)} = \sigma_{p^\beta} = \sigma_{p^\alpha} = \sigma_k$. Similarly, if $k \in A'_p$, then $\sigma_{h(k)} = \sigma_k$. Therefore, it is sufficient to prove that $h(\sigma_k) = \sigma_k$ in order to conclude, by Lemma \ref{th1h}, that $h$ is a homeomorphism. Let $x \in h(\sigma_k)$. Then, $x = h(n)$ for some $n$ such that $\gcd(n,k) = 1$. If $n \notin A_p$, then $x = h(n) = n$ and hence $x \in \sigma_k$. If $n \in A_p$, then $n = p^\alpha$ for some $\alpha$ and $x = h(n) = \sigma^p(p^\alpha) = p^\beta$ for some $\beta$, but then $\gcd(x,k) = \gcd(p^\beta, k) = \gcd(p^\alpha, k) = 1$, so $x \in \sigma_k$. If $x \in \sigma_k$, then $\gcd(x,k) = 1$. We can write $x = h(n)$ for some $n$. If $n \notin (A_p \cup A'_p)$, then $x = h(n) = n$ and $\gcd(n,k) = 1$. Thus, $x \in h(\sigma_k)$. On the other hand, if $n \in A_p$, then $x = h(n) = h(p^\alpha) = \sigma^p(p^\alpha) = p^\beta$ for some $\alpha$ and $\beta$, and thus $\gcd(n,k) = \gcd(p^\alpha, k) = \gcd(p^\beta, k) = \gcd(x,k) = 1$, so $x \in h(\sigma_k)$. Similarly, if $n \in A'_p$, then $x \in h(\sigma_k)$. Therefore, $h(\sigma_k) = \sigma_k = \sigma_{h(k)}$, and the proof is complete.
\end{proof}

Note that if in Theorem \ref{th2h} we take $p = 2$, $A_p = A'_p = \{2, 4\}$, and $\sigma^p(n)$ given by $2 \to 4$ and $4 \to 2$, then we obtain Example \ref{exam1}. Now, similarly to Example \ref{th2h}, we can obtain the following example.

\begin{example}\label{th3h}
Let $\mathcal{P}$ be a non-empty collection of prime numbers. Then $h: M(\mathbb{N}) \to M(\mathbb{N})$ given by

\begin{equation*}
    h(n) = \left\{ \begin{array}{lcc} \sigma^p(n) & \text{if} & n \in \bigcup_{p \in \mathcal{P}} A_p  \\ \\ 
    \sigma^{-p}(n) & \text{if} & n \in \bigcup_{p \in \mathcal{P}} A'_p  \\ \\
    n & \text{if} & n \notin \left( \bigcup_{p \in \mathcal{P}} A_p \cup \bigcup_{p \in \mathcal{P}} A'_p \right) \end{array} \right.
\end{equation*}
is a homeomorphism.
\end{example}

For each $n \in \mathbb{N}$, we denote by $\mathbb{P}_n$ the set of prime numbers that divide $n$. We denote by $\mathbb{P}^{\mathbb{N}}$ the set $\{p^n : p \in \mathbb{P}, n \in \mathbb{N}\}$. Finally, we leave the reader with the following theorem that summarizes the properties we found regarding self-homeomorphisms on $M(\mathbb{N})$.

\begin{theorem}\label{mainth}
Each homeomorphism $h: M(\mathbb{N}) \to M(\mathbb{N})$ has the following properties:
\begin{enumerate}
    \item $h(1) = 1$ (Lemma \ref{lem1h}).
    \item $h(\mathbb{P}^{\mathbb{N}}) = \mathbb{P}^{\mathbb{N}}$ (Lemma \ref{lem3h}).
    \item $h(\mathbb{P}_n) = \mathbb{P}_{h(n)}$ for every $n \in \mathbb{N}$ (Lemma \ref{lem5h}).
    \item $h(\sigma_n) = \sigma_{h(n)}$ for every $n \in \mathbb{N}$ (Lemma \ref{lem6h}).
\end{enumerate}
\end{theorem}

\printbibliography
\end{document}